\documentclass[12pt, reqno]{amsart}
\usepackage{amsmath, amsthm, amscd, amsfonts, amssymb, graphicx, color}
\usepackage[bookmarksnumbered, colorlinks, plainpages]{hyperref}

\input{mathrsfs.sty}

\newtheorem{theorem}{Theorem}[section]

\newtheorem{corollary}[theorem]{Corollary}
\theoremstyle{definition}

\theoremstyle{remark}
\newtheorem{remark}[theorem]{Remark}

\begin{document}

\title[Around Operator Monotone Functions]{Around Operator Monotone Functions}

\author[M.S. Moslehian]{Mohammad Sal Moslehian}

\address{Department of Pure Mathematics\\
Center of Excellence in Analysis on Algebraic Structures (CEAAS)\\
Ferdowsi University of Mashhad\\
P.O. Box 1159\\
Mashhad 91775\\
Iran}

\email{moslehian@ferdowsi.um.ac.ir and moslehian@member.ams.org}

\author[H. Najafi]{Hamed Najafi}
\address{Department of Pure Mathematics\\
Ferdowsi University of Mashhad\\
P.O. Box 1159\\
Mashhad 91775\\
Iran}
\email{hamednajafi20@gmail.com}



\begin{abstract}
We show that the symmetrized product $AB+BA$ of two positive
operators $A$ and $B$ is positive if and only if $f(A+B)\leq
f(A)+f(B)$ for all non-negative operator monotone functions $f$ on
$[0,\infty)$ and deduce an operator inequality. We also give a
necessary and sufficient condition for that the composition $f\circ
g$ of an operator convex function $f$ on $[0,\infty)$ and a
non-negative operator monotone function $g$ on an interval $(a,b)$
is operator monotone and give some applications.
\end{abstract}
\maketitle

\section{Introduction}

Let $\mathbb{B}(\mathscr{H})$ denote the algebra of all bounded linear operators on a complex Hilbert space $(\mathscr{H}, \langle \cdot,\cdot\rangle)$. An operator $ A\in \mathbb{B}(\mathscr{H})$ is called \emph{positive} if $\langle Ax, x\rangle \geq 0$ holds for every $x\in \mathscr{H}$ and then we write $A\geq 0$. The set of all positive operators on $\mathscr{H}$ is denoted by $\mathbb{B}(\mathscr{H})_+$. For self-adjoint operators $A,B \in \mathbb{B}(\mathscr{H})$, we say $A\leq B$ if $B-A\geq0$. The symmetrized product of two operators $A,B\in \mathbb{B}(\mathscr{H})$ is defined by $S(A,B)= AB+BA$. In general, the symmetrized product of two operators $A,B\in \mathbb{B}(\mathscr{H})_+$ is not positive. For example, if $A=\left(
\begin{array}{cc}1 & 0 \\ 0 & 0 \\\end{array}\right)$ and $B=\left(\begin{array}{cc}1 & 1 \\1 & 1 \\\end{array}\right)$, then $AB+BA$ is not a positive operator. The operator $\frac{1}{2}S(A,B)$ is called the \emph{Jordan product} of $A$ and $B$. Nicholson \cite{DWN} and Strang \cite{W-S} introduced some sufficient conditions for that the Jordan product of two positive matrices $A$ and $B$ is positive. Also Gustafson \cite{KG} showed that if $0 \leq m \leq A \leq M$ and $0 \leq n \leq B \leq N$, then
\begin{eqnarray}\label{eqn:1-1}
 mn - \frac{(M-m)(N-n)}{8} \leq \frac{1}{2}S(A, B)\,.
\end{eqnarray}
Throughout the paper we assume all functions to be continuous. Let
$f$ be a real-valued function defined on an interval $J$. If for
each self-adjoint operators $A, B \in \mathbb{B}(\mathscr{H})$ with
spectra in $J$,
\begin{itemize}
\item  $A\leq B$ implies $f(A)\leq f(B)$, then $f$ is called  \emph{operator monotone};\\
\item  $f(\lambda A + (1-\lambda)B) \leq \lambda f(A)+(1-\lambda)f(B)$ for all $\lambda \in [0,1]$, then $f$ is said to be \emph{operator convex};\\
\item $f(A+B)\leq f(A)+f(B)$, then $f$ is called \emph{subadditive}.
\end{itemize}
If $f$ is an operator monotone function on $[0,\infty)$, then $f$ can be represented as
$$f(t)=f(0)+\beta t + \int_0^\infty \frac{\lambda t}{\lambda+t} d\mu(\lambda)\,,$$
where $\beta \geq 0$ and $\mu$ is a positive measure on $[0,\infty)$ and if $f$ is an operator convex function on $[0,\infty)$, then $f$ can be represented as
$$f(t)=f(0)+\beta t +\gamma t^2 +\int_0^\infty \frac{\lambda t^2}{\lambda+t} \ d\mu(\lambda)\,,$$
where $\gamma \geq 0$, $\beta= f_+^{'}(0)=\lim_{t\to
0^+}\frac{f(t)-f(0)}{t}$  and $\mu$ is a positive measure on
$[0,\infty)$; see \cite[Chapter V]{BH}. If $f$ is a non-negative
operator monotone function on $[0,\infty)$, then the subadditivity
of $f$ does not hold in general. Aujla and Bourin \cite{A-B} showed
that if $A,B \geq 0$ are matrices and $f : [0,\infty)\to [0,\infty)$
is a concave function, i.e. $-f$ is convex, then there exist
unitaries $U, V$ such that
$$f(A + B)\leq Uf(A)U^* + V f(B)V^*\,.$$
Regarding the subadditivity property, Ando and Zhan \cite{A-Z} proved that $|||f (A + B)|||  \leq  |||f (A) + f (B)|||$ for all non-negative operator monotone functions $f$, all unitarily invariant norms $|||\cdot|||$ and all matrices $A, B \geq 0$; see also \cite{A-S}. Recall that a norm $|||\cdot|||$ on the algebra of all $n \times n$ matrices is unitarily invariant if $|||X|||=|||UXV|||$ for all unitaries $U$ and $V$ and all matrices $X$.

Throughout this paper $\mathcal{M}(\mathscr{H})$ denotes the set of  all $(A,B)\in \mathbb{B}(\mathscr{H})_+\times \mathbb{B}(\mathscr{H})_+$ for which $f(A+B)\leq f(A)+f(B)$ for all non-negative operator monotone functions $f$ on $[0,\infty)$. We shall show that
\begin{eqnarray*}
\mathcal{M}(\mathscr{H})= \{ (A,B)\in \mathbb{B}(\mathscr{H})_+\times \mathbb{B}(\mathscr{H})_+ | \ AB+BA \geq 0 \ \}\,.
\end{eqnarray*}
We apply this result to present an operator inequality involving operator monotone functions analogue to that of Hansen \cite{HAN}. We also give a necessary and sufficient condition for the composition $f\circ g$ of an operator convex function $f$ on $[0,\infty)$ and a non-negative operator monotone function $g$ on an interval $(a,b)$ to be operator monotone and deduce some operator inequalities.

\section{The results}
We start this section with one of our main results.

\begin{theorem}\label{1-1}
Let $A,B\in \mathbb{B}(\mathscr{H})_+$. Then $AB+BA$ is positive if and only if $f(A+B)\leq f(A)+f(B)$ for all non-negative operator monotone functions $f$ on $[0,\infty)$.
\end{theorem}

\begin{proof}
Suppose that $f(A+B)\leq f(A)+f(B)$ for all non-negative operator monotone functions $f$ on $[0,\infty)$. Let $\lambda \in (0,\infty)$. The function $f_{\lambda}(t)=\frac{\lambda t}{\lambda+t}$ is operator monotone on  $[0,\infty)$. Hence
\begin{eqnarray}\label{mos2}
(A+B)(\lambda+A+B)^{-1}\leq A(\lambda+A)^{-1}+ B(\lambda+B)^{-1}\,.
\end{eqnarray}
Put $X_{\lambda}= A(A+\lambda)^{-1}$ and $Y_{\lambda}= B(B+\lambda)^{-1}$. Inequality \eqref{mos2} is equivalent to
\begin{align*}
(\lambda+A+B)(A+B) &\leq (\lambda+A+B)X_{\lambda}(\lambda+A+B)\\
&\quad +(\lambda+A+B)Y_{\lambda}(\lambda+A+B)\\
& = \lambda^2X_{\lambda} + 2\lambda AX_{\lambda}+\lambda BX_{\lambda}+\lambda X_{\lambda}B+AX_{\lambda}A+AX_{\lambda}B\\
&\quad +  BX_{\lambda}A+BX_{\lambda}B+ \lambda^2Y_{\lambda} + 2\lambda BY_{\lambda}+\lambda AY_{\lambda}\\
&\quad +\lambda Y_{\lambda}A+BY_{\lambda}B+BY_{\lambda}A+AY_{\lambda}B+AY_{\lambda}A\,.
\end{align*}
Since $BY_{\lambda}=Y_{\lambda}B=B-\lambda Y_{\lambda}$ and $AX_{\lambda}=X_{\lambda}A=A-\lambda X_{\lambda}$ the above inequality is, in turn, equivalent to
\begin{eqnarray*}
\lambda(A+B)+A^2+B^2+AB+BA &\leq&  \lambda(A+B)+A^2+B^2\\
&& + 2(AB+BA)+BX_{\lambda}B+AY_{\lambda}A
\end{eqnarray*}
or
\begin{eqnarray}\label{eqn:4-2}
AB+BA+BX_{\lambda}B+AY_{\lambda}A \geq 0\,.
\end{eqnarray}
Letting $\lambda\to \infty $ we get $BX_{\lambda}B+AY_{\lambda}A \to
0$, whence $AB+BA \geq 0$.

Conversely, assume that $AB+BA \geq 0$. Since inequality (\ref{eqn:4-2}) holds for each $\lambda \in\mathbb{R}_+$, we obtain $f_{\lambda}(A+B)\leq f_{\lambda}(A)+f_{\lambda}(B)$.
Now, if $f$ be a non-negative operator monotone function on $[0,\infty)$, then $f$ can be represented on $[0,\infty)$ by
$$f(t)=f(0)+\beta t + \int_0^\infty f_{\lambda }(t) \ d\mu(\lambda),$$
where $\beta \geq 0$ and $\mu$ is a positive measure on $[0,\infty)$. Since $f(0)\geq 0$,
without loss of generality, we can assume that $f(t)= \int_0^\infty f_{\lambda }(t)d\mu (\lambda)$. Therefore
\begin{eqnarray*}
f(A+B)&=&\int_0^\infty  f_{\lambda}(A+B) d\mu (\lambda)\\
&\leq& \int_0^\infty  (f_{\lambda}(A)+f_{\lambda}(B)) d\mu (\lambda)\\
& = &\int_0^\infty  f_{\lambda}(A) d\mu (\lambda) + \int_0^\infty  f_{\lambda}(B) d\mu (\lambda)\\
& = &f(A)+f(B).
\end{eqnarray*}
\end{proof}

If $A,B\in \mathbb{B}(\mathscr{H})_+$ commute, then $AB=(A^{1/2}B^{1/2})^2\geq 0$. The above theorem therefore shows that $f(A+B)\leq f(A)+f(B)$ for any non-negative operator  monotone function $f$ on $[0,\infty)$. In particular, for each $n\in \mathbb{N}$ we have $f(nA)\leq nf(A)$.

\begin{corollary}\label{1-2}
Let  $0 \leq m \leq A \leq M$, $0 \leq n \leq B \leq N$ and $f$ be a non-negative operator monotone function on $[0,\infty)$. If $(M-m)(N-n) \leq 8mn$, then
$$f(A+B)\leq f(A)+f(B)\,.$$
\end{corollary}
\begin{proof}
 Use inequality \eqref{eqn:1-1}.
\end{proof}

\begin{corollary}
Let $0\leq p\leq \frac{1}{2}$ and $f$ be a non-negative operator monotone function on $[0,\infty)$. Then for positive operators $A$ and $B$ with $A \leq B$ it holds that
\begin{eqnarray*}
f\left(B^p\right) \leq f\left(\frac{B^p+A^p}{2}\right)
+ f\left(\frac{B^p-A^p}{2}\right).
\end{eqnarray*}
\end{corollary}
\begin{proof}
Assume that $S_1=\frac{1}{2}\left(B^p+A^p\right)$ and
$S_2 =\frac{1}{2}\left(B^p-A^p\right)$.  Clearly $S_1\geq 0$ and due to $f(t)=t^p$ is operator monotone \cite[Theorem 1.8]{seo}, $S_2 \geq 0$. We have
\begin{eqnarray*}
2(S_1S_2+S_2S_1)=(S_1+S_2)^2-(S_1-S_2)^2 = B^{2p}-A^{2p} \geq 0,
\end{eqnarray*}
since $f(t)=t^{2p}$ is also operator monotone. Therefore
\begin{eqnarray*}
f\left(B^p\right) = f(S_1 + S_2)  \leq  f(S_1) + f(S_2)= f\left(\frac{B^p+A^p}{2}\right)
+ f\left(\frac{B^p-A^p}{2}\right).
\end{eqnarray*}
\end{proof}

Let $m \leq M\leq (2\sqrt2 +1)m$. If $0 \leq m \leq A, B \leq M$, then inequality \eqref{eqn:1-1} ensures that $AB+BA$ is positive.
From this we can introduce an inequality for non-negative operator monotone functions. We should notice that Hansen \cite{HAN} proved that an operator monotone function $f$ on $[0, \infty)$ satisfies the inequality $C^\ast f(A)C\leq f(C^\ast AC)$ for all positive operators $A$ and all contractions $C$, i.e. operators of norm less than or equal to one.

\begin{theorem}\label{1-4}
Let $f$ be a non-negative operator monotone function on $[0,\infty)$. Let $A$ and $A_i \,\,(1\leq i\leq n)$ be positive operators with spectra in $[\lambda,(1+2\sqrt{2})\lambda]$ for some $\lambda \in \mathbb{R}_+$. Then
\begin{itemize}
\item[(i)] For every isometry $C \in \mathbb{B}(\mathscr{H})$,
$$ f(C^*AC)\leq  2 \ C^*f\left(\frac{A}{2}\right)C\,.$$

\item[(ii)] For operators  $C_i$ $(i=1, \cdots,n)$ with $\sum_{i=1}^nC_i^{*}C_i=I$,
$$ f\left(\sum_{i=1}^n C_i^*A_iC_i\right)\leq 2 \sum_{i=1}^n C_i^*f\left(\frac{A}{2}\right)C_i\,.$$
\end{itemize}
\end{theorem}
\begin{proof}
(i) Put
$X=\left(
  \begin{array}{cc}
    A & 0 \\
    0 & A \\
  \end{array}
\right)$
, $D=(I-CC^*)^{\frac{1}{2}}$, $V=\left(
  \begin{array}{cc}
    C & -D \\
    0 & C^* \\
  \end{array}
\right)$ and $U=\left(
  \begin{array}{cc}
    C & D \\
    0 & -C^* \\
  \end{array}
\right)$. It is easy to see that $U$ and $V$ are unitary operators. So $\mbox{sp}(U^*XU)= \mbox{sp}(V^*XV)= \mbox{sp}(X)\subseteq [\lambda,(1+2\sqrt{2})\lambda]$, where ``$\mbox{sp}$'' stands for spectrum. Thus
\begin{align*}
 \left(
   \begin{array}{cc}
     f(C^*AC) & 0 \\
     0 & f(DAD+CAC^*) \\
   \end{array}
 \right) &=  f\left(
              \begin{array}{cc}
                C^*AC & 0 \\
                0 & DAD+CAC^* \\
              \end{array}
            \right)
 \\ &= f\left(\frac{U^*XU+V^*XV}{2}\right) \\
 &\leq f\left(U^*\frac{X}{2} U\right)+f\left(V^*\frac{X}{2}V\right)\\
 &= U^*f\left(\frac{X}{2}\right)U+V^*f\left(\frac{X}{2}\right)V\\
 &= {\scriptsize2\left(
              \begin{array}{cc}
                C^*f\left(\frac{A}{2}\right)C & 0 \\
                0 & D f\left(\frac{A}{2}\right)D+Cf\left(\frac{A}{2}\right)C^* \\
              \end{array}
            \right),}
\end{align*}
whence $f(C^*AC)\leq \ 2 C^*f\left(\frac{A}{2}\right)C$.

(ii) Set
\begin{eqnarray*}
\widetilde{C}=\left(
          \begin{array}{c}
            C_1 \\
           \vdots \\
            C_n \\
          \end{array}
        \right), \ \ \widetilde{A}=\left(
                                \begin{array}{cccc}
                                  A_1 &  &  &  0  \\
                                  & A_2 & \ \ &  \\
                                   & \ \ \ \ \ & \ddots & \\
                                    0 &  &  & A_n \\
                                \end{array}
                              \right)\,.
\end{eqnarray*}
Thus
\begin{eqnarray*}
f\left(\sum_{i=1}^n C_i^*A_iC_i\right) = f(\widetilde{C}^*\widetilde{A}\widetilde{C}) \leq 2 \widetilde{C}^*f\left(\frac{\widetilde{A}}{2}\right)\widetilde{C}=2 \ \sum_{i=1}^n C_i^*f\left(\frac{A}{2}\right)C_i.
\end{eqnarray*}
\end{proof}

\begin{corollary}
Let $\lambda \in \mathbb{R}_+$ and $w_i \in [\lambda,(1+2\sqrt{2})\lambda)]\,\,(i=1,\cdots, n)$. Let $f$ be a non-negative operator monotone function on $[0,\infty)$ and $A_i$ be positive operators such that $\sum_{i=1}^n A_i=I$. Then
$$ f(\sum_{i=1}^n w_iA_i)\leq 2\sum_{i=1}^n f(\frac{w_i}{2})A_i\,.$$
\end{corollary}
\begin{proof}
Put $C_i = A_i^{\frac{1}{2}}$ in part (ii) of Theorem \ref{1-4}.
\end{proof}

\begin{theorem}\label{2-6}
Let $A,B \in \mathbb{B}(\mathscr{H})_+$. Then $B^2 \leq A^2$ if and
only if for each operator convex function $f$ on $[0,\infty)$ with
$f^{'}_{+}(0) \geq 0$ it holds that
$$ f(B)\leq f(A)$$
\end{theorem}
\begin{proof}
Let $A,B \in \mathbb{B}(\mathscr{H})_+$. Suppose that for each
operator convex function $f$ on $[0,\infty)$ with $f^{'}_{+}(0) \geq
0$ we have $f(B)\leq f(A)$. Due to $f(t)=t^{2}$ is operator convex
on $[0,\infty)$ and $f^{'}_{+}(0) \geq 0$, we get $B^2 \leq A^2$.

Next we show the converse. First assume that $A$ and $B$ are
invertible positive operators such that $B^{2} \leq A^{2}$. Then
$\lambda A^{-2}+A^{-1} \leq \lambda B^{-2}+B^{-1}$. This inequality
is equivalent to
$(\lambda B^{-2}+B^{-1})^{-1} \leq (\lambda A^{-2}+A^{-1})^{-1}$ and this is, in turn, equivalent to $B^{2}(B+\lambda)^{-1} \leq A^{2}(A+\lambda)^{-1}$. On the other hand $f_{\lambda}(t)=\frac{\lambda t^2}{\lambda+t}$ is operator convex on $[0,\infty)$, so $f_{\lambda}(B) \leq f_{\lambda}(A)$.\\
If $A$ and $B$ are positive and $(A-B,A+B)\in \mathcal{M}(\mathscr{H})$, then for each $\epsilon \geq 0$ let us set $A_{\epsilon}=A+\epsilon$ and $B_{\epsilon}=B+\epsilon$. Then $(A_{\epsilon}-B_{\epsilon},A_{\epsilon}+ B_{\epsilon})\in \mathcal{M}(\mathscr{H})$. Hence
$$ f_{\lambda}(B+\epsilon)\leq f_{\lambda}(A+\epsilon)\,.$$
Letting $\epsilon\to 0$ we get $ f_{\lambda}(B)\leq f_{\lambda}(A)$.

Now suppose that $f$ is an operator convex function on $[0,\infty)$ with $f^{'}_{+}(0) \geq 0$. It is known that $f$ can be represented on $[0,\infty)$ by
\begin{eqnarray}\label{mos3}
f(t)=f(0)+\beta t +\gamma t^2 + \int_0^\infty f_{\lambda }(t) \ d\mu(\lambda)\,,
\end{eqnarray}
where $\gamma \geq 0$, $\beta = f_+^{'}(0)$ and $\mu$ is a positive measure on $[0,\infty)$; see \cite[Chapter V]{BH}. If
$$f(t)= \int_0^\infty f_{\lambda }(t) \ d\mu(\lambda)\,,$$
then due to $f_{\lambda}(B)\leq f_{\lambda}(A)$ for each $\lambda \in \mathbb{R}_+$ we have
$$f(B)\leq f(A)\,.$$
Since $\gamma,\beta \geq 0$ and $B^2 \leq A^2$ the validity of $f(B)\leq f(A)$ is deduced in the general case when $f$ is given by \eqref{mos3}.
\end{proof}
\begin{remark}
It follows from the identity
\begin{eqnarray*}
(A-B)(A+B)+(A+B)(A-B)= 2 (A^2-B^2)
\end{eqnarray*}
that $B^2 \leq A^2$ if and only if $(A-B,A+B)\in
\mathcal{M}(\mathscr{H})$ .
\end{remark}

We need the L\"owner Theorem for establishing our next main result.

\begin{theorem}\cite[L\"owner Theorem]{LOE}\label{Loewner}
A function $g$ defined on $(a, b)$ is operator monotone if and only if it is analytic in $(a, b)$, can analytically be continued to the whole upper half-plane $\{z\in \mathbb{C}: Im(z)>0\}$ and represents there an analytic function whose imaginary part is non-negative.
\end{theorem}

\begin{theorem}\label{mos}
Let $g$ be a non-negative operator monotone function on an interval $(a,b)$. Let $g(z)=u(z)+ i v(z)$ be its analytic continuation to the upper half-plane. Then for each operator convex function $f$ on $[0,\infty)$ with $f^{'}_{+}(0) \geq 0$, $f\circ g$ is operator monotone on $(a,b)$ if and only if $u(z)\geq 0$ on the upper half-plane.
\end{theorem}
\begin{proof}
Assume that $u(z)\geq 0$. By the L\"owner Theorem \ref{Loewner}, $v(z)\geq 0$ on the upper half-plane. One then sees that $g$ maps the upper half-plane into the first quadrant of plane. Hence $g^2$ maps the upper half-plane into itself. Utilizing again the L\"owner Theorem we conclude that $g^2$ is operator monotone. Assume that $A$ and $B$ are self-adjoint operators with spectra in $(a,b)$ and $B\leq A$. It follows from $(g(A)-g(B))(g(A)+g(B))+(g(A)+g(B))(g(A)-g(B))= 2 (g(A)^2-g(B)^2)$ that $(g(A)-g(B), g(A)+g(B)) \in \mathcal{M}(\mathscr{H})$. Now Theorem \ref{2-6} implies that for each  operator convex function $f$ on $[0,\infty)$ with $f^{'}_{+}(0) \geq 0$ the function $f\circ g$ is operator monotone on $(a,b)$.

\noindent Conversely, assume that there would be a complex number $z_0$ with ${\rm Im} z_0>0$ such that $u(z_0)<0$. Because $g$ is operator monotone, $v(z_0)\geq 0$. Hence $\frac{\pi}{2} <{\rm Arg}(g(z_0))< \pi$. Therefore $\pi <{\rm Arg}\left(g^2(z_0)\right)< 2\pi$. So that ${\rm Im~} g^2(z_0)< 0$. Using the L\"owner Theorem \ref{Loewner}, $g^2$ is not operator monotone on $(a,b)$, which is a contradiction, since $f(x)=x^2$ is operator convex on $(0,\infty)$.
\end{proof}
\begin{remark}
If a non-negative operator monotone function $g$ on an interval $(a,b)$ is not a zero constant function and $u(z)\geq 0$ on the domain of $g$, then $u(z)>0$ on this domain. To see this note that:\\
(i) If for some $t_0 \in (a,b)$ we have $g(t_0)=0$, then there exists $R>0$ such that $g$ can be represented as $g(t)=\sum_{n=1}^{\infty}\frac{g^{(n)}(t_0)}{n!}(t-t_0)^n$ for all $t \in (t_0-R, t_0+R)$ \cite[p. 63]{B-S}. Since $g$ is a non-negative monotone function, $g(t) = 0$ for all $t \leq t_0$. Hence $g^{(n)}(t_0)=0$ for all $n$, so $g$ is zero on the neighborhood $(t_0-R, t_0+R)$ of $t_0$. Thus $\{t: g(t)=0\}$ is clopen. Hence $g$ is zero on $(a,b)$ contradicting the assumption above. Thus $g(t)>0$ on $(a,b)$.\\
(ii) If $g \neq 0$ is a constant function, then clearly $u(z)>0$.\\
(iii) If $g$ is not a constant function, then by the open mapping theorem for non-constant analytic functions, $u$ maps the upper half-plane into $\{z: {\rm Im}z >0~ \&~ u(z) > 0 \}$.
\end{remark}

\begin{corollary}\label{3-3}
Let $0\leq p \leq \frac{1}{2}$ and let $f$ be an operator convex function on $[0,\infty)$ with $f^{'}_{+}(0) \geq 0$. If $B \leq A$, then
$$f(B^p) \leq f(A^p)\,.$$
\begin{proof}
Since for $0 \leq p \leq \frac{1}{2}$ the function $g(t)=t^{p}$ is non-negative operator monotone and $g$ takes the upper half-plane into the first quarter of plane, by Theorem \ref{mos} we get $f\circ g$ is operator monotone on $[0,\infty)$.
\end{proof}
\begin{corollary}
Let $0\leq p \leq \frac{1}{2}$ and $f$ be a non-negative operator monotone function on $[0,\infty)$. Then for positive operators $A$ and $B$ with $B\leq A$,

(i) $B^pf(B^p)\leq A^pf(A^p)$;

(ii) If $f$ is strictly positive on $(0,\infty)$ and $A, B$ are invertible, then $ A^{p-1}f(A^p) \leq B^{p-1}f(B^p)$.
\end{corollary}
\begin{proof}
(i) Due to $f$ is operator monotone on $[0,\infty)$, by \cite[Theorem 2.4]{H-P}, the function $g(t)=tf(t)$ is operator convex on $[0,\infty)$, hence $g^{'}_{+}(0)=f(0) \geq 0$. Corollary \ref{3-3} then yields
$$B^pf(B^p)\leq A^pf(A^p)\,.$$
(ii) By part (i), $h(t)=t^pf(t^p)$ is operator monotone on $(0,\infty)$. By \cite[Corollary 2.6]{H-P}), $th(t)^{-1}$ is operator monotone, hence
$$B^{1-p}f(B^p)^{-1}\leq A^{1-p}f(A^p)^{-1}\,.$$
Therefore
$$ A^{p-1}f(A^p) \leq B^{p-1}f(B^p)\,.$$
\end{proof}
\end{corollary}

\textbf{Acknowledgment.} The authors would like to sincerely thank the anonymous referee for useful comments and suggestions.

\bibliographystyle{amsplain}

\end{document}